\documentclass[a4paper,10pt]{scrartcl}

\RequirePackage{enumerate}
\RequirePackage{verbatim}

\RequirePackage{amsmath,amsfonts,amssymb}
\RequirePackage[mathcal,mathbf]{euler}

\usepackage{latexsym,amscd,amsthm}
\usepackage[all]{xy}

\usepackage{amstext,amsopn}
\usepackage{graphicx}
\usepackage[dvips]{color}
\usepackage{epsfig}
\usepackage{psfrag}
\usepackage{amscd}
\usepackage{amssymb}
\usepackage{fullpage}
\usepackage{url}
\usepackage{manfnt} 
\usepackage{xcolor}
\usepackage{supertabular}
\usepackage{rotating}
\usepackage{tikz}

\def\dar[#1]{\ar@<2pt>[#1]\ar@<-2pt>[#1]}
\entrymodifiers={!!<0pt,0.7ex>+}

\newtheorem{theorem}{Theorem}[section]
\newtheorem{proposition}[theorem]{Proposition}
\newtheorem{lemma}[theorem]{Lemma}
\newtheorem{corollary}[theorem]{Corollary}      
\newtheorem{Conj}[theorem]{Conjecture}

\theoremstyle{remark}
\newtheorem{remark}{Remark}[section]

\theoremstyle{definition}
\newtheorem{definition}[theorem]{Definition}
\newtheorem{example}[theorem]{Example}
\newtheorem{notation}[theorem]{Notation}

\renewcommand\L{\mathbb{L}}
\newcommand\T{\mathbb{T}}
\newcommand\bfT{\mathbf{T}}
\newcommand\ddR{d_{dR}}
\newcommand\Gm{\mathbb{G}_m}

\author{Damien Calaque}
\title{Shifted cotangent stacks are shifted symplectic}
\date{}

\begin{document}

\maketitle

\begin{abstract}
\noindent\textbf{Abstract.} We prove that shifted cotangent stacks carry a canonical shifted symplectic structure. 
We also prove that shifted conormal stacks carry a canonical Lagrangian structure. 
These results were believed to be true, but no written proof was available in the Artin case. \\
~\\
\noindent\textbf{R\'esum\'e.} On d\'emontre que les champs cotangents d\'ecal\'es sont canoniquement munis d'une structure symplectique d\'ecal\'ee. 
On d\'emontre \'egalement que les champs conormaux d\'ecal\'es sont munis d'une structure Lagrangienne canonique. 
Ces r\'esultats \'etaient attendus mais aucune d\'emonstration n'\'etait disponible dans le cas des champs d'Artin. 
\end{abstract}

\setcounter{tocdepth}{2}
\tableofcontents

\section*{Introduction}

Shifted symplectic and Lagrangian structures have been introduced in the inspiring paper \cite{PTVV}, where many examples of shifted symplectic stacks are given. They appeared to be very powerful tools, allowing for instance: 
\begin{itemize}
\item to prove existence of symplectic structures on various moduli spaces (\cite{PTVV}). 
\item to prove existence and functoriality of perfect symmetric obstruction theories (\cite{PTVV}). 
\item to construct ``classical'' topologial field theories (\cite{Cal,Rune}). 
\end{itemize}

Many results of classical symplectic geometry do extend to the shifted/derived context, sometimes even in a better way thanks to the flexibility of derived geometry. 
For instance, Lagrangian correspondences compose well, without any transversality assumption. 
But some obvious results happen to be much more difficult in the derived context, precisely because of the flexibility of derived geometry. 
The main example of this phenomenon is the classical (and easy) correspondence between symplectic and non-degenerate Poisson structures; in the shifted/derived setting this correspondence is much harder to prove, even in the affine case (see \cite{CPT+,Prid}). 

Another classical and rather easy fact in symplectic geometry is that the cotangent $T^*Y$ to a manifold $Y$ admits a canonical 
symplectic structure and that the conormal $N^*X\subset T^*Y$ of a submanifold $X\subset Y$ is Lagrangian. 
An analog of this was expected to be true for shifted cotangent (and conormal) stacks of derived Artin stacks, but no written proof was available (the Deligne-Mumford case is treated in \cite{PTVV}). 

We provide a proof in this paper. The spirit of the proof is essentially the same as the one in usual symplectic geometry. 
The challenge is to formulate it in a completely coordinate independent fashion. 

We also conclude the paper with several conjectures about the relation between Lagrangian morphisms, (deformed) shifted cotangent stacks, and shifted Poisson structures in the sense of \cite{CPT+,Prid}.  
\subsection*{Basic notation and conventions}

\begin{itemize}
\item $\mathbf{k}$ is a field of characteristic zero (or just a Noetherian $\mathbb{Q}$-algebra). 
\item all algebro-geometric structures are over $\mathbf{k}$. For instance, a derived stack is a derived $\mathbf{k}$-stack. 
\item a \textit{cdga} means a commutative differential graded $\mathbf{k}$-algebra sitting in non-positive degree. 
For commutative differential graded algebras without any degree condition we will say \textit{unbounded cdga}. 
\item An Artin stack means a derived geometric $\mathbf{k}$-stack locally of finite presentation in the sense of \cite{HAG-II}. 
%\item For a derived stack $X$ and a quasi-coherent module $E$ on $X$, we define 
%$\mathbf{E}^*[n]:=\mathbb{R}\mathrm{Spec}_X\big(\mathrm{Sym}(E[-n])\big)$ for every $n\in\mathbb{Z}$. We obviously write 
%$\mathbf{E}^*=\mathbf{E}^*[0]$. 
%Whenever $E$ is perfect then $\mathbf{E}^*[n]$ is the derived stack of sections of the shifted dual $E^\vee[n]$ of $V$: an $A$-point of $\mathbf{E}^*[n]$, where $A$ is a cdga, is an $A$-point 
%$x:\mathbb{R}\mathrm{Spec}(A)\to X$ of $X$ together with a section of $x^*E^\vee[n]$. 
%\item In particular if $X$ a derived Artin stack then its cotangent $\L_{X}$ (relative to $\mathbf{k}$) 
%is perfect, and we denote by $\T_{X}$ its dual. 
%Hencefore we can define the $n$-shifted cotangent derived stack $\bfT^*[n]X$ of $X$ as 
%$\bfT^{*}[n]X:=\mathbb{R}\mathrm{Spec}_X\big(\mathrm{Sym}(\T_{X}[-n])\big)$. We obviously write $\bfT^*X=\mathbf{T}^*[0]X$. 
%\item similarly, we have $\bfT^*[n]f=\mathbb{R}\mathrm{Spec}_X\big(\mathrm{Sym}(\T_f[-n])\big)$ for a morphism $f:X\to Y$ of derived Artin stacks. 
\item as we work with $\infty$-categories, unless otherwise specified all usual categorical terms must be understood $\infty$-categorically (for instance, ``limit'' means ``$\infty$-limit'' or ``homotopy limit''). 
\end{itemize}

\subsection*{Acknowledgments}

I thank Pavel Safronov for several stimulating discussions about shifted cotangent stacks. 
I actually first came up with a different (and quite indirect) strategy for the proof, that uses shifted symplectic groupoids (it will appear in a forthcoming work), 
and it was after a discussion with Pavel during the conference ``Homotopical Methods in Quantum Field Theory'' at IBS-CGP in Pohang that I started to look for a simpler and more direct proof. 

I acknowledge the support of the Institut Universitaire de France and ANR SAT.  

\section{Basics of derived (pre)symplectic geometry}

\subsection{Some $\infty$-categorical linear algebra}

Let $\mathcal C$ be a stable symmetric monoidal $\infty$-category. 
Let $\mathbf{V}$ be a perfect (meaning dualizable) object in $\mathcal C$. 
\begin{definition}\label{def-ndpair}
\textbf{a)}~A degree $n$ pairing $\omega:\mathbf{V}^{\otimes2}\to\mathbf{1}[n]$, where $\mathbf{1}$ is a unit, is \emph{skew-symmetric} if it factors through $\mathbf{V}^{\otimes2}\to\wedge^2\mathbf{V}\to\mathbf{1}[n]$, where 
$\wedge^2\mathbf{V}:=S^2(\mathbf{V}[-1])[2]$. From now on, unless otherwise specified, all pairings will be skew-symmetric. \\
\textbf{b)}~A degree $n$ pairing $\omega:\mathbf{V}^{\otimes2}\to\mathbf{1}[n]$ is \emph{non-degenerate} if the adjoint morphism $\omega^\flat:\mathbf{V}\to\mathbf{V}^\vee[n]$ is an equivalence. 
\end{definition}
\begin{remark}\label{remark-duals}
Recall that there are \textit{a priori} two ways to define the adjoint morphism of $\omega$: 
\begin{itemize}
\item as $\omega^\flat:=(id_{\mathbf{V}}\otimes coev_{\mathbf{V}})\circ(\omega\otimes id_{\mathbf{V}^\vee})$, where $coev_{\mathbf{V}}:\mathbf{1}\to \mathbf{V}\otimes\mathbf{V}^\vee$ is the coevaluation map. 
\item or as $ ^\flat\omega:=(coev_{\mathbf{V}^\vee}\otimes id_{\mathbf{V}})\circ (id_{\mathbf{V}^\vee}\otimes\omega)$. 
\end{itemize}
Since $\omega$ is skew-symmetric then these two definitions coincide. Now observe that the (shifted) dual of $\omega^\flat$ is 
$(coev_{\mathbf{V}^\vee[n]}\otimes id_{\mathbf{V}})\circ (id_{\mathbf{V}^\vee[n]}\otimes\omega^\flat[-n]\otimes id_{\mathbf{V}})\circ(id_{\mathbf{V}^\vee[n]}\otimes ev_{\mathbf{V}})$, 
where $ev_{\mathbf{V}}:\mathbf{V}^\vee\otimes\mathbf{V}\to \mathbf{1}$ is the evaluation map. 
Using that $(coev_{\mathbf{V}}\otimes id_{\mathbf{V}})\circ (id_{\mathbf{V}}\otimes ev_{\mathbf{V}})$ is homotopic to $id_{\mathbf{V}}$ in $Map_{\mathcal C}(\mathbf{V},\mathbf{V})$ we get that 
the (shifted) dual of $\omega^\flat$ is homotopic to $ ^\flat\omega$ (and thus to itself) in $Map_{\mathcal C}(\mathbf{V},\mathbf{V}^\vee[n])$. 
\end{remark}
Let $f:\mathbf{U}\to\mathbf{V}$ be a morphism between perfect objects in $\mathcal C$. 
\begin{definition}\label{def-ndmap}
Let $\omega$ be a degree $n$ pairing on $\mathbf{V}$. \\
\textbf{a)}~An \emph{isotropic structure} on $f$ (w.r.t.~$\omega$) is a homotopy between $0$ and $f^*\omega$ in the space $Map_{\mathcal C}\big(\wedge^2\mathbf{U},\mathbf{1}[n]\big)$ of degree $n$ pairings on $\mathbf{U}$. \\
\textbf{b)}~An isotropic structure on $f$ is \emph{non-degenerate} if the induced null-homotopic sequence $\mathbf{U}\to\mathbf{V}\to\mathbf{U}^\vee[n]$ is a fiber sequence. 
\end{definition}
The following result is very easy, though quite useful. 
\begin{lemma}\label{lemma-nd}
Assume that we are given a morphism $f:\mathbf{U}\to\mathbf{V}$ between perfect objects together with an $n$-shifted pairing $\omega$ on $\mathbf{V}$ and an isotropic structure $\gamma$ on $f$. 
If $\gamma$ is non-degenerate then $\omega$ is non-degenerate as well. 
\end{lemma}
\begin{proof}
By definition, the null-homotopic sequence $\mathbf{U}\to\mathbf{V}\to\mathbf{U}^\vee[n]$ is a fiber sequence. 
Its (shifted) dual $\mathbf{U}\to\mathbf{V}^\vee[n]\to\mathbf{U}^\vee[n]$ is thus a fiber sequence as well. 
The main point is that, thanks to Remark \ref{remark-duals}, these two sequences are mapped to each other in the following way: 
$$
\xymatrix{
\mathbf{U} \ar[r]\ar@{=}[d] & \mathbf{V}   \ar[d]^{\omega^\flat}\ar[r] & \mathbf{U}^\vee[n] \ar@{=}[d] \\
\mathbf{U}       \ar[r] & \mathbf{V}^\vee[n] \ar[r] & \mathbf{U}^\vee[n]
}
$$
Hence the morphism $\omega^\flat:\mathbf{V}\to \mathbf{V}^\vee[n]$ is an equivalence and $\omega$ is non-degenerate. 
\end{proof}

\subsection{Shifted (pre)symplectic, isotropic and Lagrangian structures}\label{sec-presymp}

Recall from \cite{PTVV} that to any derived stack $X$ one can associate a graded mixed ($\mathbf{k}$-)module $\mathbf{DR}(X)$. 
We will often denote $d_{dR}$ the mixed differential $\epsilon$ (of weight $1$ and cohomological degree $1$)\footnote{We use the weight/grading convention from \cite{CPT+} rather than the one of \cite{PTVV}). }, 
and $d_{int}$ or $d_X$ the differential of the underlying graded module $\mathbf{DR}(X)^\sharp$ (of weight $0$ and cohomological degree $1$), also called the \textit{internal differential}. 

If $X=\mathrm{Spec}(A)$, for a cofibrant cdga $A$, then $\mathbf{DR}(X)=\mathbf{DR}(A/\mathbf{k})$ is the graded module 
$$
Sym_A\big(\Omega^1_A[-1](-1)\big)\,.
$$
equipped with the following mixed differential: 
$$
\epsilon=d_{dR}:a_0(da_1)\dots (da_n)\mapsto (da_0)(da_1)\dots (da_n)\,.
$$
Note that $(da_i)$ is \textit{a priori} not ``$d$ applied to $a_i$ and is in particular \textbf{not} $d_{int}a_i=d_Aa_i$. 
But observe that $(da_i)=d_{dR}a_i$; hence we will try to use the notation $d_{dR}a_i$ in order to avoid confusion. 
\begin{notation}
We write $\epsilon-\mathbf{dg}_{\mathbf{k}}^{gr}$, resp.~$\mathbf{dg}_{\mathbf{k}}^{gr}$, for the category of graded mixed modules $\mathbf{k}$-modules, resp.~graded $\mathbf{k}$-modules. 
We refer to \cite[Section 1]{CPT+} for details about the homotopy theory of graded mixed modules. 
\end{notation}
\begin{definition}
\textbf{a)}~The space of \emph{$p$-forms of degree $n$ on $X$} is 
$$
A^p(X,n):=Map_{\mathbf{dg}_{\mathbf{k}}^{gr}}\big(\mathbf{k},\mathbf{DR}(X)^\sharp[n+p](p)\big)\,.
$$ 
\textbf{b)}~The space of \emph{closed $p$-forms of degree $n$ on $X$} is 
$$
A^{p,cl}(X,n):=Map_{\epsilon-\mathbf{dg}_{\mathbf{k}}^{gr}}\big(\mathbf{k},\mathbf{DR}(X)[n+p](p)\big)\,.
$$ 
\textbf{c)}~For a closed $p$-form of $\alpha$ degree $n$, we denote by $\alpha_0$ its \emph{underlying $p$-form of degree $n$}, 
that is to say its image through the forgetful functor $(-)^\sharp$. 
\end{definition}
The functors $A^p(n):=A^p(-,n)$ and $A^{p,cl}(n):=A^{p,cl}(-,n)$ are actually derived stacks (see \cite[Proposition 1.11]{PTVV}). 
\begin{remark}
Note that for a given graded mixed module $M$, $Map_{\epsilon-\mathbf{dg}_{\mathbf{k}}^{gr}}(\mathbf{k},M)$ is the space of a semi-infinite sequences 
$$
\alpha=(\alpha_0,\dots,\alpha_i,\cdots)
$$
such that $d(\alpha_0)=0$, and for every $i\geq0$, $\epsilon(\alpha_i)=d(\alpha_{i+1})$. 
I.e.~it is the space of $0$-cocycles in $\prod_{k\geq0}M^{(p)}$ equipped with the total differential $d+\epsilon$. 
The image of such a sequence $\alpha$ through the forgetful functor $(-)^\sharp$ is precisely its leading term $\alpha_0$, which is a $0$-cocyle in $M^{(0)}$. 
We will thus sometimes write, in the context of the above definition, ``\emph{the leading term}'' in place of ``the underlying $p$-form of degree $n$'' of a closed $p$-form of degree $n$. 
\end{remark}
\begin{definition}
\textbf{a)}~An \emph{$n$-shifted presymplectic structure} on a stack $X$ is a closed $2$-form of degree $n$ on $X$. \\
\textbf{b)}~Let $(X,\omega)$ be a stack equipped with an $n$-shifted presymplectic structure. 
An \emph{isotropic structure} on a morphism $f:L\to X$ is a path from $0$ to $f^*\omega$ in the space $A^{2,cl}(L,n)$ of $n$-shifted presymplectic structures on $L$. \\
\textbf{c)}~We say that an $n$-shifted presymplectic structure $\omega$ on an $X$ is \emph{symplectic} if $X$ has a perfect cotangent complex $\mathbb{L}_X$ and if the degree $n$ pairing on $\mathbb{T}_X:=\mathbb{L}_X^\vee$ 
given by its leading term $\omega_0$ is non-degenerate in the sense of Definition \ref{def-ndpair}. \\
\textbf{d)}~We say that an isotropic structure $\gamma$ on $f:L\to X$ is \emph{non-degenerate} if both $X$ and $L$ have perfect cotangent complexes and if the leading term $\gamma_0$ is indeed a non-degenerate isotropic structure 
on the morphism $\mathbb{T}_L\to f^*\mathbb{T}_X$ in the sense of Definition \ref{def-ndmap}. It is only when the presymplectic structure $\omega$ itself is symplectic that we name \emph{Lagrangian structures} the 
non-degenerate isotropic ones. 
\end{definition}

The main result of Section \ref{sec-symplcotan} below is that the $n$-shifted cotangent stack $\bfT^*[n]X$ of an Artin stack $X$ carries an $n$-shifted symplectic structure and that the zero section morphism $X\to\bfT^*[n]X$ carries 
a Lagrangian structure. 

\medskip

We refer to \cite{PTVV,Cal,Saf} for many other examples of $n$-shifted symplectic and Lagrangian structures. 

\begin{remark}
\textbf{a)}~A stack equipped an $n$-shifted presymplectic structure (an \textit{$n$-shifted presymplectic stack}, for short) is nothing but a stack over $A^{2,cl}(n)$. \\
\textbf{b)}~ Let $(X,\omega)$ be a stack equipped with an $n$-shifted presymplectic structure. 
An \emph{isotropic structure} on a morphism $f:L\to X$ is nothing but the data of a commuting square 
$$
\xymatrix{
L \ar[r]^{f} \ar[d] & X \ar[d]^{\omega} \\
\mathrm{*} \ar[r]^{\!\!\!\!\!\!\!\!\!\!\!\!0} & A^{2,\mathrm{cl}}(n)
}
$$
\end{remark}

The above alternative approach to isotropic structures on morphisms (due to Rune Haugseng \cite[Section 9]{Rune}) leads to the following: 
\begin{definition}
An \textit{$n$-isotropic correspondence} is a correspondence in the $\infty$-category $\mathbf{dSt}_{/A^{2,cl}(n)}$ of derived stacks over $A^{2,cl}(n)$: it is the data of a commuting square 
$$
\xymatrix{
L \ar[r]^{g} \ar[d]_{f} & Y \ar[d]^{\omega_Y} \\
X \ar[r]_{\!\!\!\!\!\!\!\!\!\!\!\!\omega_X} & A^{2,\mathrm{cl}}(n)
}
$$ 
\end{definition}
In other words, an $n$-isotropic correspondence is the data of two $n$-shifted presymplectic stacks $(X,\omega_X)$ and $(Y,\omega_Y)$ together with an isotropic morphism $f\times g:L\to \overline{X}\times Y$, 
where $\overline{X}\times Y$ is the $n$-shifted presymplectic stack $(X\times Y,\pi_2^*\omega_Y-\pi_1\omega_X)$. Assuming $\omega_X$ and $\omega_Y$ are symplectic, then $f\times g$ is Lagrangian if and only if the 
commuting square 
$$
\xymatrix{
\mathbb T_L \ar[r] \ar[d] & g^*\mathbb T_Y\simeq g^*\mathbb L_Y[n] \ar[d] \\
g^*\mathbb T_X\simeq f^*\mathbb{L}_X[n] \ar[r] & \mathbb L_L[n]
}
$$
is (co)Cartesian. In this case we talk about an \textit{$n$-Lagrangian correspondence}. 

\subsection{Isotropic and Lagrangian fibrations}

Recall from \cite[Section 1]{CPT+} that to any morphism $f:X\to Y$ of derived stacks one can associate a graded mixed module $\mathbf{DR}(f):=\mathbf{DR}(X/Y)$. 
%We borrow the notation from Subsection \ref{sec-presymp} for the mixed and the internal differential. 
If $f:X=\mathrm{Spec}(A)\to \mathrm{Spec}(B)=Y$ comes from a cofibration $B\to A$ between cofibrant cdgas, then $\mathbf{DR}(X/Y):=\mathbf{DR}(A/B)$ is the graded module 
$$
Sym_A\big(\Omega^1_{A/B}[-1](-1)\big)\,.
$$
equipped with the mixed differential: 
$$
\epsilon:a_0(da_1)\dots (da_n)\mapsto (da_0)(da_1)\dots (da_n)\,.
$$
\begin{definition}
\textbf{a)}~The space of \emph{relative $p$-forms of degree $n$ on $f$} (a-k-a \emph{$p$-forms of degree $n$ on $X$ relative to $Y$}) is 
$$
A^p(f,n)=A^p(X/Y,n):=Map_{\mathbf{dg}_{\mathbf{k}}^{gr}}\big(\mathbf{k},\mathbf{DR}(X/Y)^\sharp[n+p](p)\big)\,.
$$ 
\textbf{b)}~The space of \emph{relative closed $p$-forms of degree $n$ on $f$} is 
$$
A^{p,cl}(f,n)=A^{p,cl}(X/Y,n):=Map_{\epsilon-\mathbf{dg}_{\mathbf{k}}^{gr}}\big(\mathbf{k},\mathbf{DR}(X/Y)[n+p](p)\big)\,.
$$ 
\textbf{c)}~For a relative closed $p$-form of $\alpha$ degree $n$, we denote by $\alpha_0$ its \emph{underlying relative $p$-form of degree $n$}, 
that is to say its image through the forgetful functor $(-)^\sharp$. 
\end{definition}
Note that we have a null-homotopic sequence 
$$
\mathbf{DR}(Y)\overset{f^*}{\longrightarrow}\mathbf{DR}(X)\overset{~_{/Y}}{\longrightarrow} \mathbf{DR}(X/Y)\,.
$$
\begin{remark}\label{rem-sheafDR}
Recall from \cite[Remark 2.4.8]{CPT+} that $\mathbf{DR}(X/Y)$ can be obtained as global sections of a stack $\mathcal{DR}(f)=\mathcal{DR}(X/Y)$ of graded mixed modules 
over the $\infty$-site of derived affine schemes over $Y$: 
$$
\mathcal{DR}(X/Y)\big(\mathrm{Spec}(A)\overset{y}{\longrightarrow} Y\big):=\mathbf{DR}\big(y^*X/\mathrm{Spec}(A)\big)\,.
$$
Observe that $\mathcal{DR}(X/Y)$ is actually a graded mixed $\mathcal O_Y$-module as every $\mathbf{DR}\big(y^*X/\mathrm{Spec}(A)\big)$ is a graded mixed $A$-module. 
\end{remark}
Let now $X$ be equipped with an $n$-shifted presymplectic structure $\omega$. 
\begin{definition}
\textbf{a)}~An \emph{isotropic fibration structure on $f$} is a path from $0$ to $\omega_{/Y}$ in the space $A^{2,cl}(X/Y,n)$ of relative closed $2$-forms of degree $n$. \\
\textbf{b)}~An isotropic fibration structure $\gamma$ on $f$ is \emph{non-degenerate} (or is a \emph{Lagrangian fibration structure}) if both $X$ and $Y$ have perfect cotangent complexes 
and if the leading term $\gamma_0$ is indeed a non-degenerate isotropic structure on the morphism $\mathbb{T}_{X/Y}\to \mathbb{T}_X$ in the sense of Definition \ref{def-ndmap}. 
\end{definition}
\begin{remark}
It follows from Lemma \ref{lemma-nd} that if $\gamma$ is non-degenerate then the $n$-shifted presymplectic structure $\omega$ is indeed $n$-shifted symplectic. 
Therefore we don't need here to make a distinction between non-degenerate isotropic fibration structures and Lagrangian fibration structures. 
\end{remark}
We will see in Section \ref{sec-symplcotan} that the projection $\bfT^*[n]X\to X$ of a shifted cotangent stack (of an Artin stack) comes equipped with a Lagrangian fibration structure. 

\section{The $n$-shifted symplectic structure of the $n$-shifted cotangent stack}\label{sec-symplcotan}

Let $X$ be an Artin stack. 
It has a perfect global \textit{cotangent complex} $\L_X$ (see \cite{HAG-II}), its dual $\T_X:=\L_X^\vee$ being called the \textit{tangent complex}. 
Our main object of interest is its $n$-shifted cotangent stack $\bfT^*[n]X:=\mathbb{R}\mathrm{Spec}_X\big(\mathrm{Sym}(\T_X[-n])\big)$, that is nothing but the stack of $n$-shifted $1$-forms on $X$. 
Note that, given a quasi-coherent module $E$ over $X$, the stack of sections $\mathbf{E}:=\mathbb{A}_X(E)$ is defined, as a stack over $X$, by 
$$
\mathbf{E}\big(\mathrm{Spec}(A)\overset{x}{\to} X\big):=Map_{A-mod}(A,x^*E)\,.
$$
Whenever $E$ is perfect, we get that this is equivalent to 
$$
Map_{A-mod}(x^*E^\vee,A)\cong Map_{A-alg}\big(\mathrm{Sym}_A(x^*E^\vee),A\big)=:\mathbb{R}\mathrm{Spec}_X\big(\mathrm{Sym}(E^\vee)\big)\big(\mathrm{Spec}(A)\overset{x}{\to} X\big)\,,
$$
and one can show that $\mathbb{R}\mathrm{Spec}_X\big(\mathrm{Sym}(E^\vee)\big)$ is an Artin stack as well (it is an affine stack of finite presentation over $X$; see \cite{To} for what we mean by an \textit{affine stack}). 

\medskip

In this Section we prove that the $n$-shifted cotangent stack $\bfT^*[n]X$ is $n$-shifted symplectic, together with variations on this result. 
All this was mainly already known for Deligne-Mumford stacks. 

\subsection{An $n$-shifted presymplectic structure on $\bfT^*[n]X$}

Recall that the space of morphisms $Y\to \bfT^*[n]X$ is nothing but the space of morphisms $f:Y\to X$ equipped with a section of $f^*\L_X[n]$. 
In particular the identity map $\bfT^*[n]X\to \bfT^*[n]X$ corresponds to the data of the projection map $\pi_X:\bfT^*[n]X\to X$ together with a section of $\pi_X^*\L_X[n]$. 
Therefore, using the map $\pi_X^*\L_X[n]\to\L_{\bfT^*[n]X}[n]$ we obtain that $\bfT^*[n]X$ carries a tautological $n$-shifted $1$-form $\lambda_X$. 
The morphism $\ddR:A^1(-,n)\to A^{2,cl}(-,n)$ sends it to an $n$-shifted closed $2$-form $\omega_X=\ddR\lambda_X$ on $\bfT^*[n]X$. 

\medskip

The zero section morphism $\iota_X:X\to \bfT^*[n]X$ corresponds to the data of the identity map $X\to X$ together with the zero section of $\L_X[n]$. 
Tautologically, this means that the pull-back of $\lambda_X$ along the zero section is homotopic to zero. 
Hence we obtain an isotropic structure $\gamma_X$ on $\iota_X$ (w.r.t.~the $n$-shifted presymplectic structure previously described). 

\medskip

Finally observe $\lambda_X$ comes from a section of $\pi_X^*\L_X[n]$ and recall that the sequence 
$$
\pi_X^*\L_X\to \L_{\bfT^*[n]X}\to \L_{\pi_X}
$$
is fibered. 
Therefore the image of $\lambda_X$ into relative $n$-shifted $1$-forms along $\pi_X$ is homotopic to zero. 
Hence we obtain that the projection morphism $\pi_X$ is equipped with an isotropic fibration structure $\eta_X$. 
\begin{remark}\label{rem-identification}
Note that $\L_{\pi_X}$ is nothing but $\pi_X^*\T_X[-n]$. 
\end{remark}

\subsubsection{Compatibility with the $\Gm$-action}

Observe that, for a perfect module $E$ over $X$, its stack of sections $\mathbf{E}$ is acted on by $\Gm$ because $\mathrm{Sym}(E^\vee)$ is a graded $\mathcal O_X$-algebra (by symmetric powers). 
We will refer to this new grading as the \textit{fiber grading} (not to confuse it with the weight of forms and closed forms). 
Both the zero section $X\to \mathbf{E}$ and the projection $\mathbf{E}\to X$ are $\Gm$-equivariant for the trivial $\Gm$-action on $X$. 

\begin{notation}\label{notation}
\textbf{a)}~If $Y$ is an Artin stack equipped with an action of $\Gm$, then the relative cotangent complex $\L_{[Y/\Gm]/B\Gm}$ is a quasi-coherent sheaf on $[Y/\Gm]$. 
Its pull-back along the quotient map $Y\to [Y/\Gm]$ identifies with $\L_Y$. 
We will therefore allow ourselves to abuse notation and to keep writing $\L_Y$ for its $\Gm$-equivariant enhancement $\L_{[Y/\Gm]/B\Gm}$. 
Similarly, if $f:Y_1\to Y_2$ is a $\Gm$-equivariant map, then we will keep writing $\L_f$ for its $\Gm$-equivariant enhancement $\L_{[f]}$, 
where $[f]:[Y_1/\Gm]\to [Y_2/\Gm]$ is the induced map on quotient stacks. \\
\textbf{b)}~Sheaves on $B\Gm=[*/\Gm]$ are identified with graded complexes. 
For a graded complex $M=(M_i)_{i\in\mathbb{Z}}$ we denote by $M\{k\}$ its $k$-th shift for this grading (meaning that $M\{k\}_i=M_{i+k}$). 
Without further precision, a genuine complex is always assumed to be equipped with the trivial $\Gm$-action. 
This notation extends as well to sheaves on $[X/\Gm]=X\times B\Gm$, that are nothing but graded sheaves on $X$. 
\end{notation}
\begin{remark}\label{rem-gridentification}
With this notation the identification from Remark \ref{rem-identification} becomes $\L_{\pi_X}\cong\pi_X^*\T_X[-n]\{-1\}$. 
\end{remark}

Now, if $Y$ is a derived stack equipped with an action of $\Gm$, observe that a $\Gm$-equivariant map $Y\to\mathbf{E}$ is determined by 
the data of a $\Gm$-equivariant map $f:Y\to X$ together with a $\Gm$-equivariant section of $f^*E\{1\}$. 
Applying this to the identity map $\bfT^*[n]X\to \bfT^*[n]X$ we get a $\Gm$-equivariant section of $\pi_X^*\L_X\{1\}$. 
As a consequence, the $1$-form $\lambda_X\in A^1\big(\bfT^*[n]X,n\big)\cong \Gamma\big(\bfT^*[n]X,\L_{\bfT^*[n]X}[n]\big)$ of (cohomological) degree $n$ 
is homogeneous of weight $1$ for the fiber grading. 

\medskip

Recall from Remark \ref{rem-sheafDR} that for a stack $Y$ equipped with a $\Gm$-action we have a sheaf $\mathcal{DR}\big([Y/\Gm]/B\Gm\big)$ of graded mixed $\mathcal O$-modules on $B\Gm$. 
Its pull-back along the canonical point in $B\Gm$ is $\mathbf{DR}(Y)$. This means that $\mathbf{DR}(Y)$ carries yet another auxiliary grading that enhances it to a graded mixed graded $\mathbf{k}$-module. 
We will follow the (abusing) Notation \ref{notation} and keep writing $\mathbf{DR}(Y)$ for its graded enhancement $\mathcal{DR}\big([Y/\Gm]/B\Gm\big)$. 
Similarly, if $f:Y_1\to Y_2$ is a $\Gm$-equivariant map, then $\mathbf{DR}(f)$ also admits $\Gm$-equivariant enhancement that we still denote the same. \\

\medskip

Going back to $\bfT^*[n]X$ we get that $\lambda_X$ belongs to the space 
$$
A^{1,\{1\}}\big(\bfT^*[n]X,n\big):=Map_{(\mathbf{dg}_{\mathbf{k}}^{gr})^{gr}}\big(\mathbf{k},\mathbf{DR}(\bfT^*[n]X)^\sharp[n+1](1)\{1\}\big)
$$
and that $\omega_X=\ddR\lambda_X$ belongs to the space 
$$
A^{2,cl,\{1\}}\big(\bfT^*[n]X,n\big):=Map_{(\epsilon-\mathbf{dg}_{\mathbf{k}}^{gr})^{gr}}\big(\mathbf{k},\mathbf{DR}(\bfT^*[n]X)[n+2](2)\{1\}\big)
$$
of closed $2$-forms of (cohomological) degree $n$ on $\bfT^*[n]X$ that are homogeneous of weight $1$ for the fiber grading. 
Similarly, $\eta_X$ is a path in the space 
$$
A^{2,cl,\{1\}}\big(\pi_X,n\big):=Map_{(\epsilon-\mathbf{dg}_{\mathbf{k}}^{gr})^{gr}}\big(\mathbf{k},\mathbf{DR}(\pi_X)[n+2](2)\{1\}\big)
$$
of relative closed $2$-forms of (cohomological) degree $n$ along $\pi_X$ that are homogeneous of weight $1$ for the fiber grading.

\subsection{Non-degeneracy of the $n$-shifted presymplectic structure}

We will now prove that $\omega_X$ actually defines an $n$-shifted symplectic structure on $\bfT^*[n]X$, that is to say the morphism 
$$
\T_{\bfT^*[n]X}\longrightarrow \L_{\bfT^*[n]X}[n]
$$
induced by its underlying $n$-shifted $2$-form is an equivalence. 

\medskip 

We also prove that the isotropic morphism $X\to\bfT^*[n]X$ and the isotropic fibration $\bfT^*[n]X\to X$ are Lagrangian. 
\begin{theorem}\label{theorem-cotangent}
The $n$-shifted cotangent stack $\bfT^*[n]X$ is $n$-shifted symplectic, its zero section is Lagrangian, and the projection down to $X$ is a Lagrangian fibration. 
More precisely: 
\begin{enumerate}
\item The $n$-shifted presymplectic structure $\omega_X$ on $\bfT^*[n]X$ is non-degenerate. 
\item The $n$-shifted isotropic structure $\gamma_X$ on the zero section $\iota_X:X\to\bfT^*[n]X$ is non-degenerate. 
\item The $n$-shifted isotropic fibration structure $\eta_X$ on $\pi_X:\bfT^*[n]\to X$ is non-degenerate. 
\end{enumerate}
\end{theorem}

Before we start it is important to have in mind the two following fiber sequences, that are ($n$-shifted) dual to each other: 
$$
\pi_X^*\L_X[n]\cong \T_{\pi_X}\to \T_{\bfT^*[n]X}\to \pi_X^*\T_X \qquad \textrm{and} \qquad \pi_X^*\L_X[n]\to \L_{\bfT^*[n]X}[n]\to \L_{\pi_X}[n]\cong \pi_X^*\T_X\,.
$$
Moreover, when pulled-back along the zero section these sequences split and we get that 
$$
\iota_X^*\T_{\bfT^*[n]X}\cong\L_X[n]\oplus\T_X \qquad \textrm{and} \qquad \iota_X^*\L_{\bfT^*[n]X}[n]\cong\L_X[n]\oplus\T_X\,.
$$
\begin{remark}
The above fiber sequences actually are $\Gm$-equivariant. More precisely, according to Remark \ref{rem-gridentification} we thus get fiber sequences 
$$
\pi_X^*\L_X[n]\{1\}\to \T_{\bfT^*[n]X}\to \pi_X^*\T_X \qquad \textrm{and} \qquad \pi_X^*\L_X[n]\to \L_{\bfT^*[n]X}[n]\to \L_{\pi_X}[n]\cong \pi_X^*\T_X\{-1\}
$$
of $\Gm$-equivariant sheaves on $\bfT^*[n]X$ that split when pulled-back along the zero section: 
$$
\iota_X^*\T_{\bfT^*[n]X}\cong\L_X[n]\{1\}\oplus\T_X \qquad \textrm{and} \qquad \iota_X^*\L_{\bfT^*[n]X}[n]\cong\L_X[n]\oplus\T_X\{-1\}\,.
$$
\end{remark}

\subsubsection{The isotropic structure on the zero section $\iota_X$ is non-degenerate}\label{subsubiotand}

We prove here the second part of Theorem \ref{theorem-cotangent}	. 

\medskip

We have an isotropic structure $\gamma_X$ on the zero section morphism $\iota_X:X\to T^*[n]X$, inducing in particular a commuting diagram 
$$
\xymatrix{
\T_X             \ar[d]\ar[r] & \T_X                \ar[d]\ar[r] & 0 \ar[d] \\
\iota_X^*\T_{\bfT^*[n]X} \ar[r] & \iota_X^*\L_{\bfT^*[n]X}[n] \ar[r] & \L_X[n] 
}
$$
in which the right-hand square is Cartesian (this square is the equivalence $\iota_X^*\L_{\bfT^*[n]X}[n]\cong\L_X[n]\oplus\T_X$). 
The arrow $\T_X\to \T_X$ in the left-most square is determined by the leading term of $\gamma_X$ (which is a homotopy between the composed map $T_X\to\L_X(n]$ and the zero map). 

We will now show that 
\begin{lemma}\label{lemma-phi}
The morphism $\phi_X:\T_X\to \T_X$ appearing in the above diagram is natural in $X$, in the sense that for every morphism $f:X\to Y$ we get a commuting square
$$
\xymatrix{
\T_X \ar[d]\ar[r]^{\phi_X}    & \T_X    \ar[d] \\
f^*\T_Y    \ar[r]^{f^*\phi_Y} & f^*\T_Y 
}
$$
\end{lemma}
\begin{corollary}\label{cor-phi-equiv}
The isotropic structure on the zero section $\iota_Y:Y\to T^*[n]Y$ is non-degenerate, that is to say the map $\phi_Y$ is an equivalence. 
\end{corollary}
\begin{proof}[Proof of the Corollary]
We know that it is true for derived affine schemes locally of finite presentation. 
Let us then assume by induction that we know the result holds for $m$-Artin stacks (i.e.~Artin stacks that are $m$-geometric in the sense of \cite{HAG-II}). 
Let us then prove that it is true as well for $(m+1)$-Artin stacks. 
If $Y$ is an $(m+1)$-Artin stack then $Y=|G_*|:=colim_n(G_n)$, where $G_*=(G_n)_{n\geq0}$ is a smooth Segal groupoid in $m$-Artin stacks. 
Let $q:X=G_0\to Y$ be the natural map to the colimit. 
Note that $q^*\T_Y$ can be computed as the colimit of the simplicial diagram $(e_n^*\T_{G_n})_{n\geq0}$, 
where $e_n:G_0\to G_n$ is the unit map\footnote{
Recall that $QCoh(|G_\bullet|)\cong lim_n\big(QCoh(G_n)\big)$ and that $q^*\cong lim_n(e_n^*)$. 
Under the first identification we have that $\L_{|G_\bullet|}$ is the diagram $(\L_{G_n})_{n\geq0}$ and thus, using the second identification, $q^*\L_{|G_\bullet|}$ is equivalent to $lim_n(e_n^*\L_{G_n})$. 
Dualizing, we get that $q^*\T_{|G_\bullet|}$ is equivalent to $colim_n(e_n^*\T_{G_n})$. 
}.
%and that the morphism $\pi^*\L_Y\to \L_X$ identifies with the natural map $holim_n(e_n^*\L_{G_n})\to L_{G_0}$. 

Thanks to the above Lemma we have a simplicial diagram $(e_n^*\phi_{G_n})_{n\geq0}$ of morphisms of perfect modules on $X$, 
the colimit of which is the morphism $q^*\phi_Y$. 
Thanks to the induction hypothesis every $e_n^*\phi_{G_n}$ is an equivalence, and thus $q^*\phi_Y$ is an equivalence as well. 
One finally gets that $\phi_Y$ is an equivalence because $q^*$ is conservative ($q$ is a groupoid quotient morphism). 

We get the result by induction on the degree of geometricity. 
\end{proof}

\begin{remark}\label{rem-phi-equiv}
One can even prove that the equivalence $\phi_X$ is homotopic to the identity (this is true for derived schemes, and then the inductive proof is exactly the same). 
\end{remark}

\begin{proof}[Proof of the Lemma]
Let $f:X\to Y$ be a map of Artin stacks and consider the diagram
$$
\xymatrix{
         &  X            \ar[dl]\ar[r]^{f}\ar[d] & Y \ar[d]_{\iota_Y} \\
\bfT^*[n]X &  f^*\bfT^*[n]Y\ar[l]\ar[r]         & \bfT^*[n]Y}
$$
where $f^*\bfT^*[n]Y:=\bfT^*[n]Y\times_YX$ is the stack of sections of $f^*\L_Y[n]$. 

The morphism ~$\phi_X$, resp.~$\phi_Y$, is determined by the leading term of the isotropic structure $\gamma_X$, resp.~$\gamma_Y$.  
Hence the morphism $f^*\phi_Y$ is determined by the leading term of $f^*\gamma_Y$, considered as an isotropic structure on 
$$
f^*\iota_Y:X\longrightarrow f^*\bfT^*[n]Y
$$
and where we equip $f^*\bfT^*[n]Y$ with the $n$-shifted presymplectic structure $\tilde\omega$ given by the pull-back of $\omega_Y$ along $f^*\bfT^*[n]Y\to\bfT^*[n]Y$. 
Now the result follows from the following fact: 
\begin{itemize}
\item the pull-back of the pair $(\omega_X,\gamma_X)$ along $f^*\bfT^*[n]Y\to\bfT^*[n]X$ is naturally homotopic to the pair $(\tilde\omega,f^*\gamma_Y)$. 
\end{itemize}
Actually, the pull-back of $\lambda_X$ itself along $f^*\bfT^*[n]Y\to\bfT^*[n]X$ is naturally homotopic to the pull-back of $\lambda_Y$ along 
$f^*\bfT^*[n]Y\to\bfT^*[n]Y$\footnote{This is a slight variation on the obvious fact that for any morphism $u:A\to B$, $id_B\circ u$ and $u\circ id_A$ are equivalent.}. 
All diagrams we have written down live under $X$, therefore all previously considered homotopies between $0$ and pull-backs of $\lambda$'s on $X$ are themselves equivalent. 
The Lemma is proved. 
\end{proof}

\subsubsection{The isotropic fibration structure on the projection $\pi_X$ is non-degenerate}

We now give the proof of the third part of the theorem. 
Its structure is very similar to the one of the second part (apart from that we will have to use $\Gm$-equivariant enhancements at some point). 

\medskip

We have an isotropic fibration structure $\eta_X$ on the projection $\pi_X:\bfT^*[n]X\to X$, inducing in particular a commuting diagram 
$$
\xymatrix{
\pi_X^*\L_X[n] \ar[d]\ar[r] & \pi_X^*\L_X[n] \ar[d]\ar[r] & 0 \ar[d] \\
\T_{\bfT^*[n]X}    \ar[r] & \L_{\bfT^*[n]X}[n] \ar[r] & \pi_X^*\T_X 
}
$$
in which the right hand square is Cartesian. 

Let us write $\psi_X:\pi_X^*\L_X\to\pi_X^*\L_X$ for the $(-n)$-shift of the morphism $\pi_X^*\L_X[n]\to \pi_X^*\L_X[n]$ appearing in the above diagram. 
Our aim is to prove that $\psi_X$ is an equivalence. It follows from the next Lemma that it is sufficient to prove that $\iota^*X\psi_X:\L_X\to\L_X$ is an equivalence. 
\begin{lemma}\label{lem-psi}
$\psi_X=\pi_X^*\iota_X^*\psi_X$. 
\end{lemma}
\begin{proof}
Recall that, up to a shift, the morphism $\psi_X$ is the composition of the map $\T_{\pi_X}\to\pi_X^*\L_X[n]$ given by the (leading term of the) isotropic fibration structure $\eta_X$ 
with the canonical identification $\pi^*\L_X[n]\cong\T_{\pi_X}$ (dual to the one from Remark \ref{rem-identification}). 
Taking into account $\Gm$-equivariant enhancements, it becomes the composition of: 
\begin{itemize}
\item the map $\T_{\pi_X}\to\pi_X^*\L_X[n]\{1\}$ given by (the leading term of the) path $\eta$ that is homogeneous of weight $1$ for the fiber grading, 
\item with the identification $\pi^*\L_X[n]\{1\}\cong\T_{\pi_X}$ (dual to the one from Remark \ref{rem-gridentification}). 
\end{itemize}
Hence $\psi_X$ is a $\Gm$-equivariant map $\pi_X^*\L_X\{1\}\to \pi_X^*\L_X\{1\}$. 
Since $\L_X\{1\}$ is a graded sheaf of pure weight on $X$, then $\psi_X$ is the pull-back along $\pi_X$ of 
a morphism $\L_X\to\L_X$ (that has to be $\iota_X^*\psi_X$).  
\end{proof}

Let now $f:X\to Y$ be a map of Artin stacks and consider the diagram
$$
\xymatrix{
\bfT^*[n]X \ar[dr]_{\pi_X} & \ar[l]_{T^*f} f^*\bfT^*[n]Y \ar[d]^{f^*\pi_Y}\ar[r]^{\pi_Y^*f} & \bfT^*[n]Y \ar[d]^{\pi_Y} \\
                   &  X                         \ar[r]^{f} & Y  
}
$$
We will show that 
\begin{lemma}\label{lem-naturalpsi}
The morphism $\psi_X:\pi_X^*\L_X\to \pi_X^*\L_X$ is natural in $X$, in the sense that for every morphism $f:X\to Y$ we get a commuting square
$$
\xymatrix{
(f^*\pi_Y)^*\L_X          \ar[r]^{(T^*f)^*\psi_X}    & (f^*\pi_Y)^*\L_X   \\
(f^*\pi_Y)^*f^*\L_Y \ar[u]\ar[r]^{(\pi_Y^*f)^*\psi_Y} & (f^*\pi_Y)^*f^*\L_Y  \ar[u] 
}
$$
In particular we have a commuting square 
$$
\xymatrix{
\L_X          \ar[r]^{\iota_X^*\psi_X}    & \L_X   \\
f^*\L_Y \ar[u]\ar[r]^{f^*\iota_Y^*\psi_Y} & f^*\L_Y  \ar[u] 
}
$$
\end{lemma}
\begin{corollary}\label{cor-psi-equiv}
The isotropic fibration structure on the projection $\pi_X:T^*[n]X\to X$ is non-degenerate, that is to say the map $\psi_X$ is an equivalence. 
\end{corollary}
\begin{proof}[Proof of the Corollary]
It follows from Lemma \ref{lem-psi} that it is sufficient to show that $\iota_X^*\psi_X$ is an equivalence, the proof of which is almost identical to the one 
of Corollary \ref{cor-phi-equiv}. Let us repeat it for the sake of clarity. 

We know that it is true for derived affine schemes locally of finite presentation. 
Let us then assume by induction that we know the result holds for $m$-Artin stacks and prove that it does for $(m+1)$-Artin stacks as well. 
If $Y$ is an $(m+1)$-Artin stack then $Y=|G_*|$, where $G_*$ is a smooth Segal groupoid in $m$-Artin stacks. 
Let $q:X=G_0\to Y$ be the natural map to the colimit. 
Note that $q^*\L_Y$ can be computed as the limit of the cosimplicial diagram $(e_n^*\L_{G_n})_{n\geq0}$, 
where $e_n:G_0\to G_n$ is the unit map.
%and that the morphism $\pi^*\L_Y\to \L_X$ identifies with the natural map $holim_n(e_n^*\L_{G_n})\to L_{G_0}$. 

Thanks to the above Lemma we have a cosimplicial diagram $(e_n^*\iota_{G_n}^*\psi_{G_n})_{n\geq0}$ of morphisms of perfect modules on $X$, 
the limit of which is the morphism $q^*\iota_Y^*\psi_Y$. 
Thanks to the induction hypothesis every $e_n^*\iota_{G_n}^*\psi_{G_n}$ is an equivalence, and thus $q^*\iota_Y^*\psi_Y$ is an equivalence as well. 
One finally gets that $\iota_Y^*\psi_Y$ is an equivalence because $q^*$ is conservative ($q$ is a groupoid quotient morphism). 

We get the result by induction on the degree of geometricity. 
\end{proof}

\begin{remark}
Just as in Remark \ref{rem-phi-equiv} one can even prove that the equivalence $\psi_X$ is homotopic to the identity. 
\end{remark}

\begin{proof}[Proof of the Lemma]
The morphism ~$\psi_X$, resp.~$\psi_Y$, is determined by the leading term of the isotropic fibration structure $\eta_X$, resp.~$\eta_Y$.  
Hence the morphism $(\pi_Y^*f)^*\psi_Y$ is determined by the leading term of $(\pi_Y^*f)^*\eta_Y$, considered as an isotropic fibration structure on 
$$
f^*\pi_Y:f^*\bfT^*[n]Y\longrightarrow X
$$
and where we equip $f^*\bfT^*[n]Y$ with the $n$-shifted presymplectic structure $(\pi_Y^*f)^*\omega_Y$. 
Now the result follows from the following fact: 
\begin{itemize}
\item the pull-back of $(T^*f)^*\eta_X$ is naturally homotopic to $(\pi_Y^*f)^*\eta_Y$. 
\end{itemize}
First of all this claim makes sense as we have already seen that the analogous statement for $\omega$'s is true (see the proof of Lemma \ref{lemma-phi}). 
It is even true for $\lambda$'s. Since the part of the above diagram we are interested in lives over $X$, all images of previously considered homotopies between $0$ and images of $\lambda$'s in the space of $1$-forms relative to $X$ are themselves equivalent. The Lemma is proved. 
\end{proof}

\begin{proof}[End of the proof of Theorem \ref{theorem-cotangent}]
The first part of the Theorem follows from the third one. Namely, Lemma \ref{lemma-nd} tells us that if we have a non-degenerate isotropic fibration on an $n$-shifted presymplectic Artin stack then it is $n$-shifted symplectic.  
\end{proof}

\subsection{A Lagrangian structure on shifted conormal stacks}

Let $f:X\to Y$ be a morphism of Artin stacks. 
We have already seen that there is a tautological $n$-shifted $1$-form on $\bfT^*[n]Y$, mainly determined by the identity morphism of $\bfT^*[n]Y$ (and the map $\pi_Y^*\L_Y[n]\to \L_{\bfT^*[n]Y}$). 
We now want to consider the morphism of Artin stacks $\ell_f:\bfT^*_X[n]Y\to \bfT^*[n]Y$, where $\bfT^*_X[n]Y:=\bfT^*[n-1]f=\mathbb{A}(\L_f[n-1])$ is the \textit{$n$-shifted conormal stack} of $f$. 
First of all observe that this morphism factors through $f^*\bfT^*[n]Y$, so that $\ell_f^*\omega_Y$ is homotopic to the pull-back of $\omega_X$ along the sequence
$$
\bfT^*_X[n]Y\to f^*\bfT^*[n]Y\to \bfT^*[n]X\,.
$$
Hence $\ell_f^*\omega_Y$ is equivalent to $0$ as the composed morphism of the above sequence factors through the zero section $X\to\bfT^*[n]X$. 
Let us denote $\omega_f$ the isotropic structure we have just constructed on $\ell_f$. 

Our aim is to prove the following: 
\begin{theorem}\label{theorem-lagrangian}
The $n$-shifted conormal stack is Lagrangian, that is to say the isotropic structure $\omega_f$ is non-degenerate. 
\end{theorem}
Let us provide two rather extreme examples: 
\begin{itemize}
\item consider $f:X\to pt$, where $pt$ is the terminal stack. 
In this case $\bfT^*_X[n+1]pt\cong \bfT^*[n]X$ has a Lagrangian morphism to the point equipped with its canonical $(n+1)$-shifted symplectic structure. 
We get back in a rather involved way the $n$-shifted symplectic structure on an $n$-shifted cotangent stack.
\item consider the identity morphism $X\to X$. In this case $\bfT^*_X[n]X\cong X$ and we get back that the zero section is Lagrangian. 
\end{itemize}

\subsubsection{Isotropic squares}

Our aim is to apply a strategy similar to what we did in the previous Section. 
To do so we need to exhibit a structure on the projection $\bfT^*_X[n]Y\to X$ that resembles to the one of a Lagrangian fibration. 
This is what we do in this Subsection. 

\begin{definition}
An \emph{$n$-shifted pre-isotropic square} is commuting square of Artin stacks 
$$
\xymatrix{
\mathbf{E} \ar[r]^g\ar[d] & \mathbf{F} \ar[d] \\
X \ar[r]^f       & Y
}
$$
together with the following additional structures: 
\begin{itemize}
\item an $n$-shifted presymplectic structure $\omega$ on $\mathbf{F}$: $\omega\in A^{2,cl}(\mathbf{F},n)$. 
\item an isotropic structure $\gamma$ on $\mathbf{E}\to \mathbf{F}$: $\gamma\in Path_{g^*\omega,0}\big(A^{2,cl}(\mathbf{E},n)\big)$. 
\item an isotropic fibration structure $\eta$ on $\mathbf{F}\to Y$: $\eta\in Path_{\omega_{/Y},0}\big(A^{2,cl}(\mathbf{F}/Y,n)\big)$. 
\end{itemize}
\end{definition}

Note that composing $\gamma_{/X}$ with $(g^*\eta)^{-1}$ one gets an element $\omega'\in\Omega_0A^{2,cl}(\mathbf{E}/X,n)\cong A^{2,cl}(\mathbf{E}/X,n-1)$ 

\begin{example}\label{example-preiso}
There is an $n$-shifted pre-isotropic square structure with $\mathbf{F}=\bfT^*[n]Y$, $\mathbf{E}=\bfT^*_X[n]Y$, $g=\ell_f$, $\omega=\omega_Y$, $\gamma=\omega_f$ and $\eta=\eta_Y$. 
\end{example}
\begin{definition}
Borrowing the above notation, an \emph{$n$-shifted isotropic square} is an $n$-shifted pre-isotropic square equipped with a homotopy $\xi$ in $A^{2,cl}(\mathbf{E}/X,n)$ between the two isotropic fibration structures $\gamma_{/X}$ and $g^*\eta$ on $\mathbf{E}\to X$. 
\end{definition}

We let the reader prove the following: 
\begin{proposition}
The $n$-shifted pre-isotropic square from Example \ref{example-preiso} can be upgraded to an $n$-shifted isotropic square. \hfill$\Box$
\end{proposition}

\subsubsection{Non-degeneracy of the isotropic structure on the $n$-shifted conormal stack}

Assume we are given an isotropic square. Borrowing the previous notation, we set $\mathbf{U}:=\T_{\mathbf{E}}$, $\mathbf{V}:=g^*\T_{\mathbf{F}/Y}$, $\mathbf{W}:=g^*\T_{\mathbf{F}}$ and $\mathbf{Z}:=\T_{\mathbf{E}/X}$. 
Recall that we have a lot of structure: 
\begin{itemize}
\item a degree $n$ non-degenerate pairing on $\mathbf{W}$ (given by $\omega$). 
\item a path from $0$ to its pull-back along $\mathbf{U}\to\mathbf{W}$ (given by $\gamma$). 
\item a non-degenerate path from $0$ to its pull-back along $\mathbf{V}\to\mathbf{W}$ (given by $\eta$). 
\item these two paths compose to a self-homotopy of $0$ in the space of degree $n$ pairings on $\mathbf{U}\times_{\mathbf{W}}\mathbf{V}$, defining a degree $n-1$ pairing $\beta$ on it. 
\item a path from $0$ to the pull-back of this last pairing along $\mathbf{Z}\to \mathbf{U}\times_{\mathbf{W}}\mathbf{V}$ (given by $\xi$).
\end{itemize}

\begin{lemma}\label{lemma-ndnd}
If $\xi$ is non-degenerate, then so are $\beta$ and $\gamma$
\end{lemma}
\begin{proof}
We have already seen that if $\xi$ is non-degenerate then so is $\beta$ (this is Lemma \ref{lemma-nd}). Now the data of $\omega$, $\gamma$ and $\eta$ provides us with the following commuting diagram 
$$
\xymatrix{
\mathbf{U}\times_{\mathbf{W}}\mathbf{V} \ar[d]\ar[r] & \mathbf{V} \ar[d]\ar[r] & 0 \ar[d] \\
\mathbf{U} \ar[d]\ar[r] & \mathbf{W} \ar[d]\ar[r] & \mathbf{V}^\vee[n] \ar[d] \\
0 \ar[r] & \mathbf{U}^\vee[n] \ar[r] & (\mathbf{U}\times_{\mathbf{W}}\mathbf{V})^\vee[n]
}
$$
We observe that the following squares are Cartesian: top-left (by definition), top-right (by non-degeneracy of $\eta$), bottom-right (by definition), and the big one (by non-degeneracy of $\beta$). 
Therefore the bottom-left square is Cartesian as well, which means that $\gamma$ is non-degenerate. 
\end{proof}

\begin{proof}[End of the proof of Theorem \ref{theorem-lagrangian}]
Note that, if $\mathbf{E}=\mathbb{A}(E)$, $\mathbf{F}=\mathbb{A}(F)$ and $g$ is linear in the fiber (i.e.~$g$ comes from a morphism $E\to f^*F$) then we have that the cofiber of $\mathbf{Z}\to \mathbf{U}\times_{\mathbf{W}}\mathbf{V}$ is 
$$
cofib\big(\pi_X^*E\to \T_{\mathbf{E}}\times_{g^*\T_{\mathbf{F}}}g^*\T_{\mathbf{F}/Y}\big)\cong \pi_X^*\T_{X/Y}=\pi_X^*\T_f\,.
$$

This is typically what happens in the situation of Example \ref{example-preiso}. 
In this case the null-homotopic sequence $\mathbf{Z}\to \mathbf{U}\times_{\mathbf{W}}\mathbf{V}\to \mathbf{Z}^\vee[n-1]$ then takes the following form: 
$$
\pi_X^*\L_f[n-1]\longrightarrow {\mathbf{U}}\times_{\mathbf{W}}\mathbf{V}\longrightarrow \pi_X^*\T_f\,.
$$
Hence we get a map $\pi_X^*\T_f\cong hocofib\big(\mathbf{Z}\to \mathbf{U}\times_{\mathbf{W}}\mathbf{V}\big)\to \pi_X^*\T_f$, which is an equivalence if and only if $\xi$ is non-degenerate. 
As before, one can show that this map must be an equivalence by reasoning along the following steps: 
\begin{enumerate}
\item using $\Gm$-equivariant enhancements, one proves (as in the proof of Lemma \ref{lem-psi}) that the map $\pi_X^*\mathbb{T}_f\to \pi_X^*\T_f$ is the pull-back along $\pi_X^*$ of a map $\psi_f:\T_f\to \T_f$. 
\item one then proves (along the lines of Lemmas \ref{lemma-phi} and \ref{lem-naturalpsi}) that $\psi_f$ is natural in the following sense: for any morphism commuting square 
$$
\xymatrix{
X' \ar[d]_{x}\ar[r]^{f'}    & Y'    \ar[d]^{y} \\
X  \ar[r]^{f} & Y 
}
$$ one has a commuting square 
$$
\xymatrix{
\T_{f'} \ar[d]\ar[r]^{\psi_{f'}}    & \T_{f'}    \ar[d] \\
x^*\T_f    \ar[r]^{x^*\psi_f} & x^*\T_f 
}
$$
\item the previous step allows one to reduce to the case when both $X$ and $Y$ are affine, by induction on the degree of geometricity of $f$ (as in Corollaries \ref {cor-phi-equiv} and \ref{cor-psi-equiv}). 
\begin{remark}
We say that $f$ is \textit{$m$-geometric} if it can be realized as the colimit $f=|f_*|$ of a simplicial diagram $f_*:X_*\to Y_*$ of morphisms, with both $X_*$ and $Y_*$ being Segal groupoids in $(m-1)$-geometric stacks. 
Every morphism between geometric stacks is $m$-geometric for some $m$. Namely, $Y$ being geometric there exists a Segal groupoid $Y_*$ such that $Y=|Y_*|$ and every $Y_n$ is geometric. 
Then there exists a Segal groupoid $X_*$ presenting $X$ such that $X_n\cong X\times_YY_n$ and a simplicial morphism $f_*$ presenting $f$ such that $f_n$ is equivalent to the second projection (see \cite[\S1.3.5]{HAG-II}).  
Observe that every $X_n$ is geometric, as geometricity is preserved by fiber products. 
\end{remark}
\item one finally proves that the result is true when both $X$ and $Y$ are affine. 
\end{enumerate}
We conclude by using Lemma \ref{lemma-ndnd}: since $\xi$ is non-degenerate then so is $\gamma=\omega_f$. 
\end{proof}

\begin{remark}\label{remark-cool}
We won't do it but one can also prove, using similar methods, that given two morphisms $X_1\to Y$ and $X_2\to Y$ there is an equivalence 
$$
\bfT^*_{X_1}[n]Y\times_{\bfT^*[n]Y}\bfT^*_{X_2}[n]Y\cong \bfT^*[n-1](X_1\times_YX_2)
$$
of $(n-1)$-shifted symplectic stacks. 
\end{remark}

\begin{example}
The example we provide here is just an illustration of our results, as in this case everything can be directly proven in an \textit{ad hoc} manner. 
Let $G$ be an affine algebraic group, $H\subset G$ a closed algebraic subgroup, and let $\mathfrak h\subset\mathfrak g$ be the corresponding inclusion of Lie algebras. 
We then consider the map $f:X=BH\to BG=Y$. 

Theorem \ref{theorem-cotangent} tells us that the $1$-shifted cotangent stack $T^*[1]BG=[\mathfrak g^*/G]$ is $1$-shifted symplectic 
(we refer to \cite[\S1.2.3]{Cal}, \cite[Example 2.10]{Cal2} or \cite[\S2.2]{Saf} where an explicit formula for the $1$-shifted symplectic form is given), and the zero section map 
$BG=[pt/G]\to[\mathfrak g^*/G]=T^*[1]BG$ is Lagrangian. 
Then Theorem \ref{theorem-lagrangian} says that $\ell_f:T_{BH}^*[1]BG=[\mathfrak h^\perp/H]\to[\mathfrak g^*/G]=T^*[1]BG$ is Lagrangian as well. 
It is interesting to not that in the case when $H=1$ then $\mathfrak h^perp=\mathfrak g^*$ and one recovers the known fact that $\mathfrak g^*\to [\mathfrak g^*/G]$ is Lagrangian. 

Applying Remark \ref{remark-cool} to $X_1=BH$, $X_2=pt$ and $Y=BG$ we get a chain of equivalences 
$$
\mathfrak h^\perp\times G/H\cong[\mathfrak h^\perp/H]\times_{[\mathfrak g^*/G]}\mathfrak g^*=\bfT^*_{X_1}[n]Y\times_{\bfT^*[n]Y}\bfT^*_{X_2}[n]Y\cong\bfT^*(X_1\times_YX_2)\cong\bfT^*(G/H)
$$
that recovers the obvious symplectomorphism $\mathfrak h^\perp\times G/H\cong\bfT^*(G/H)$. 
\end{example}

\subsection{Lagrangian structures from closed $1$-forms}

Let $X$ be an Artin stack and let $\alpha_0\in A^1(X,n)$ be a $1$-form of degree $n$ on $X$. It tautologically defines a section $X\overset{\bar\alpha_0}{\longrightarrow}\bfT^*[n]X$ of $\pi_X:\bfT^*[n]X\to X$. 
Our aim in this Subsection is to prove the following generalization of the second assertion in Theorem \ref{theorem-cotangent}: 
\begin{theorem}
Every lift of $\alpha_0$ to a closed form $\alpha\in A^{1,cl}(X,n)$ induces a Lagrangian structure on the morphism $\bar\alpha_0$. 
\end{theorem}
\begin{proof}[Sketch of proof of the Theorem]
By definition of the tautological $1$-form $\lambda_X$ of degree $n$ on $\bfT^*[n]X$, we have that $\bar\alpha_0^*\lambda_X=\alpha_0$. 
Let us denote $Keys(\alpha_0)$ the space of lifts of $\alpha_0$ to $A^{1,cl}(X,n)$: we have the Cartesian square
$$
\xymatrix{Keys(\alpha_0) \ar[r]\ar[d] & A^{1,cl}(X,n) \ar[d] \\
\mathrm{*} \ar[r]^{\alpha_0} & A^1(X,n)\,.
}$$
\begin{lemma}\label{lemma-keys}
There is a map from $Keys(\alpha_0)$ to space $Isot(\bar\alpha_0,\omega_X)$ of isotropic structures on $\bar\alpha_0$. 
\end{lemma}
\begin{proof}[Proof of the Lemma]
Observe that for any graded mixed complex $(E,\epsilon)$ the sequence 
$$
Map_{\epsilon-\mathbf{dg}_{\mathbf{k}}^{gr}}(\mathbf{k},E)\overset{(-)^\sharp}{\longrightarrow} Map_{\mathbf{dg}_{\mathbf{k}}^{gr}}(\mathbf{k},E^\sharp)\overset{\epsilon}{\longrightarrow} Map_{\epsilon-\mathbf{dg}_{\mathbf{k}}^{gr}}(\mathbf{k},E[1](1))
$$
is null-homotopic. 
Applying this to $E=DR(X)[n+1](1)$ we get the null-homotopic sequence 
$$
A^{1,cl}(X,n)\overset{(-)^\sharp}{\longrightarrow} A^{1}(X,n)\overset{d_{dR}}{\longrightarrow} A^{2,cl}(X,n)\,,
$$
and we therefore have commuting squares 
$$
\xymatrix{
A^{1,cl}(X,n) \ar[d]\ar[r]^{(-)^\sharp} & A^{1}(X,n) \ar[d]^{d_{dR}} & \mathrm{*} \ar[l]_{~~~~~\alpha_0}\ar[d] \\
\mathrm{*}\ar[r]^{0} & A^{2,cl}(X,n) & \mathrm{*}\ar[l]^{~~~~~d_{dR}\alpha_0}\,.
}
$$
Hence $Keys(\alpha_0)$ maps to the space of paths between $0$ and $d_{dR}\alpha_0\cong \bar\alpha_0^*\omega_X$ in $A^{2,cl}(X,n)$. 
\end{proof}
\noindent It remains to prove that the map $Keys(\alpha_0)\to Isot(\bar\alpha_0,\omega_X)$ from Lemma \ref{lemma-keys} takes its values in the subspace of non-degenerate ones. 
The proof is almost identical to the case of the zero section $\iota_X$ (see \S\ref{subsubiotand}). 
\end{proof}

\section{Shifted cotangent stacks in derived Poisson geometry}

In this Section we provide several conjectural relations between shifted cotangent stacks and shifted Poisson structures (\cite{CPT+,Prid}). 
% and coisotropic structures (\cite{CPT+,MS}). 

\subsection{The $n$-shifted Poisson structure of $\bfT^*[n]X$ and shifted polyvector fields}

On the one hand we know from \cite{CPT+,Prid} that the spaces of $n$-shifted symplectic structures and of non-degenerate $n$-shifted Poisson structures 
(we refer to \textit{loc. cit.} for a precise definition, and follow the notation from \cite{CPT+}) on a given Artin stack are equivalent. 
We therefore obtain an $n$-shifted Poisson structure on $\bfT^*[n]X$, and thus in particular a $\mathbb{P}_{n+1}$-algebra structure extending the natural unbounded cdga structure of 
$$
\Gamma\big(\bfT^*[n]X,\mathcal O_{\bfT^*[n]X}\big)=\Gamma\big(X,\mathrm{Sym}(\T_X[-n])\big)\,.
$$
Moreover, recall that $\bfT^*[n]X$ carries a $\Gm$-action, which makes $\Gamma\big(X,\mathrm{Sym}(\T_X[-n])\big)$ a commutative algebra object in graded complexes (the weight being obviously given by the symmetric degree). 
We have seen that the $n$-shifted symplectic structure on $\bfT^*[n]X$ is homogeneous of weight $1$ for this $\Gm$-action. 

\medskip

One can carry over the $\Gm$-action along \cite{CPT+} (this ends up working relatively to $B\Gm$, and the constructions from \textit{loc. cit.} do extend to the relative setting) and get 
a $\mathbb{P}_{n+1}^{gr}$-algebra structure on $\Gamma\big(X,\mathrm{Sym}(\T_X[-n])\big)$. Recall that $\mathbb{P}_{n+1}^{gr}$ for a graded version of $\mathbb{P}_{n+1}$ where the Poisson bracket generator has weight $-1$. 

\medskip

We believe the following is true: 
\begin{Conj}\label{conjecture-poisson}
The above $\mathbb{P}_{n+1}^{gr}$-algebra structure on $\Gamma\big(X,S(\T_X[-n])\big)$ is equivalent to the $\mathbb{P}_{n+1}^{gr}$-algebra 
$\mathbf{Pol}(X,n)$ of $n$-sifted polyvector fields from \cite{CPT+}. 
\end{Conj}

\subsection{Lagrangian morphisms, deformed shifted cotangent stacks and shifted Poisson structures}

Let $f:X\to Y$ be a morphism of Artin stacks equipped with an $n$-shifted Lagrangian structure $\gamma$ (in particular, $Y$ implicitly carries an $n$-shifted symplectic structure $\omega$). 
Note that the leading term $\gamma_0$ of $\gamma$ provides us with an equivalence $f^*\T_f[1]\to \L_X[n]$, which leads to an equivalence of derived stacks $\bfT_XY:=\mathbb{A}(f^*\T_f[1])\to\bfT^*[n]$. 

\medskip

We then consider the formal completion $\widehat{Y}_f=X_{DR}\times{Y_{DR}}Y$ of $Y$ along $f$. The morphism $f$ factors as $X\overset{\widehat{f}}{\to}\widehat{Y}_f\to Y$, and one can easily see that 
\begin{itemize}
\item the pull-back $\widehat{\omega}$ of $\omega$ along $\widehat{Y}_f\to Y$ is still non-degenerate, thus defining an $n$-shifted symplectic structure on $\widehat{Y}_f$. 
\item the homotopy $\gamma$ between $f^*\omega=\widehat{f}^*\widehat{\omega}$ and $0$, seen as an isotropic structure on $\widehat{f}$, is still non-degenerate, thus defining an $n$-shifted Lagrangian structure on $\widehat{f}$. 
\end{itemize}
As a matter of notation, if $E$ is a perfect module over $X$ then we simply write $\widehat{\mathbf{E}}$ for the formal completion of $\mathbf{E}$ along the zero section. 
It follows from \cite[Chapter IV.5]{GR} that there actually exists a morphism $\widetilde{f}:X\times[\mathbb{A}^1/\Gm]\to \widetilde{Y}_f$ of derived stacks over $[\mathbb{A}^1/\Gm]$ such that 
\begin{itemize}
\item the fiber at $0$ is the zero section $X\to \widehat{\bfT}_XY\cong\widehat{\bfT}^*[n]X$. 
\item the fiber at $1$ is $\widehat{f}:X\to \widehat{Y}_f$. 
\end{itemize}
As a consequence we get that there is a filtration on the unbounded cdga $\Gamma(\widehat{Y}_f,\mathcal O_{\widehat{Y}_f})$ with associated graded being 
$\Gamma\big(\bfT^*[n]X,\mathcal O_{\bfT^*[n]X}\big)=\Gamma\big(X,\mathrm{Sym}(\T_X[-n])\big)$. In other words, we get that $\Gamma\big(X,\mathrm{Sym}(\T_X[-n])\big)$ becomes equipped with a mixed differential. 
\begin{remark}
Note that there is a much easier way to see this. Thanks to the equivalence $f^*\T_f[1]\to \L_X[n]$ provided by $\gamma_0$, we get an equivalence 
$$
\Gamma\big(X,\mathrm{Sym}(\T_X[-n])\big)\to \Gamma\big(X,\mathrm{Sym}(f^*\L_f[-1])\big)\cong\mathbf{DR}(X/Y)^\sharp
$$
of commutative algebra objects in graded complexes. The mixed differential we get on $\Gamma\big(X,\mathrm{Sym}(\T_X[-n])\big)$ then comes from the mixed differential on $\mathbf{DR}(X/Y)$. 
\end{remark}

We believe the following is true: 
\begin{Conj}\label{conjecture-filtered}
\textbf{a)}~There exists a relative $n$-shifted symplectic structure $\widetilde{\omega}$ on $\widetilde{Y}_f$ over $[\mathbb{A}^1/\Gm]$ such that $1^*\widetilde{\omega}\cong\widehat{\omega}$ and 
$0^*\widetilde{\omega}\cong\omega_X$. \\
\textbf{b)}~There exists a relative Lagrangian structure $\widetilde{\gamma}$ on $\widetilde{f}$ (over $[\mathbb{A}^1/\Gm]$) such that $1^*\widetilde{\gamma}\cong\gamma$ and $0^*\widetilde{\gamma}\cong\gamma_X$. 
\end{Conj}
Assuming the above conjecture and working out the constructions from \cite{CPT+} relatively to $[\mathbb{A}^1/\Gm]$ we get that the mixed differential on $\Gamma\big(X,\mathrm{Sym}(\T_X[-n])\big)$ is compatible 
with the $\mathbb{P}_{n+1}^{gr}$-algebra. 
We believe that the following is also true: 
\begin{Conj}\label{conjecture-MC}
The above mixed differential is given by the bracket with a weight $2$ Maurer--Cartan element $\pi_f$ in the graded dg Lie algebra $\Gamma\big(X,\mathrm{Sym}(\T_X[-n])\big)[n]$. 
\end{Conj}

Putting together Conjectures \ref{conjecture-poisson}, \ref{conjecture-filtered} and \ref{conjecture-MC} we obtain a map from $n$-shifted Lagrangian morphisms $f:X\to Y$ to $(n-1)$-shifted Poisson structures $\pi$ on $X$. 
We don't expect this map $f\mapsto\pi_f$ to be an equivalence as the construction only depends on $\widehat{f}$ (namely, $\pi_f=\pi_{\widehat{f}}$). But we expect the following 
\begin{Conj}
The map $f\mapsto\pi_f$ provides an equivalence between: 
\begin{itemize}
\item the space of $n$-shifted Lagrangian morphisms $f:X\to Y$ such that $\widehat{Y}_f\to Y$ is an equivalence. 
\item the space of $(n-1)$-shifted Poisson structures on $X$ in the sense of \cite{CPT+}.  
\end{itemize}
\end{Conj}
\begin{remark}
Note that Valerio Melani and Pavel Safronov show in \cite[Section 6]{MS} that there is an equivalence between $n$-shifted Lagrangian structures on $f$ and non-degenerate $n$-shifted coisotropic structures on $f$. 
Moreover, there is a forgetful map from $n$-shifted coisotropic structures on $f$ to $(n-1)$-shifted Poisson structures on $X$. 
Therefore the work of Melani--Safronov \cite{MS} already provides a (non-conjectural) way to produce a map from $n$-shifted Lagrangian morphisms $f:X\to Y$ to $(n-1)$-shifted Poisson structures $\pi$ on $X$, without assuming any conjecture. The main advantage of the approach that we propose here is that one can make sense of everything for formal completions as well in a rather obvious way. 
\end{remark}

\end{document}